\documentclass[a4paper,12pt]{amsart}
\usepackage{standalone}
\usepackage{amsfonts}
\usepackage{amsthm}
\usepackage{amssymb}
\usepackage{amsmath}
\usepackage{theoremref}
\usepackage{enumerate}
\usepackage{bbm}
\usepackage{bm}
\usepackage{pgfplots}
\pgfplotsset{compat=1.18} 
\usepgfplotslibrary{fillbetween}
\usetikzlibrary{intersections}
\usepackage{mathtools}
\usetikzlibrary{patterns}

\usepackage{geometry}
\usepackage{a4wide}

\numberwithin{equation}{section}

\newcommand{\A}{\mathcal{A}}
\renewcommand{\L}{\mathcal{L}}

\newcommand{\M}{\mathcal{M}}

\newcommand{\N}{\mathcal{N}}

\newcommand{\T}{\mathbb{T}}

\newcommand{\conj}[1]{\overline{#1}}
\newcommand{\D}{\mathbb{D}}
\newcommand{\Di}{\mathcal{D}}
\newcommand{\B}{\mathcal{B}}

\newcommand{\Po}{\mathcal{P}}
\newcommand{\BMOA}{\textbf{BMOA}}

\newcommand{\VMOA}{\textbf{VMOA}}

\newcommand{\ip}[2]{\big\langle #1, #2 \big\rangle}

\newcommand{\m}{\textit{m}}
\newcommand{\hil}{\mathcal{H}} 
\newcommand{\hd}{Hol(\D)}

\newcommand{\hb}{\mathcal{H}(b)}

\newcommand{\To}[1]{\mathcal{T}_{\conj{#1}}}

\newtheorem{mainthm}{Theorem}

\newtheorem{thm}{Theorem}[section]
\newtheorem*{thm*}{Theorem}
\newtheorem{lem}[thm]{Lemma}

\newtheorem{cor}[thm]{Corollary}
\newtheorem*{cor*}{Corollary}
\newtheorem{prop}[thm]{Proposition}
\theoremstyle{definition}
\newtheorem{example}[thm]{Example}
\theoremstyle{definition}

\newtheorem{defn}[thm]{Definition}

\newtheorem*{question*}{Question}

\newtheorem{claim*}{Claim}

\title{Embeddings into de Branges-Rovnyak spaces}
\author[Malman and Seco]{Bartosz Malman and Daniel Seco}

\address{Division of Mathematics and Physics, 
        Mälardalen University,
		Västerås, Sweden}
\email{bartosz.malman@mdu.se}

\address{Universidad de la Laguna e IMAULL,  Avenida Astrof\'isico Francisco S\'anchez, s/n, Departamento de An\'alisis Matem\'atico, 38206 San Crist\'obal de La Laguna, Santa Cruz de Tenerife,  Spain} \email{dsecofor@ull.edu.es}

\begin{document}

\begin{abstract} We study conditions for containment of a given space $X$ of analytic functions on the unit disk $\D$ in the de Branges-Rovnyak space $\hb$. We deal with the non-extreme case in which $b$ admits a Pythagorean mate $a$, and derive a multiplier boundedness criterion on the function $\phi = b/a$ which implies the containment $X \subset \hb$. With our criterion, we are able to characterize the containment of the Hardy space $\hil^p$ inside $\hb$, for $p \in [2, \infty]$. The end-point cases have previously been considered by Sarason, and we show that in his result, stating that $\phi \in \hil^2$ is equivalent to $\hil^\infty \subset \hb$, one can in fact replace $\hil^\infty$ by $\BMOA$. We establish various other containment results, and study in particular the case of the Dirichlet space $\Di$, containment of which is characterized by a Carleson measure condition. In this context, we show that matters are not as simple as in the case of the Hardy spaces, and we carefully work out an example.
\end{abstract}

\thanks{Part of this work has been done during the first author's visit at the Department of Mathematics and Statistics of Université Laval in Québec, Canada, funded by the Simons-CRM Scholar-in-Residence program. The second author is funded through grant PID2023-149061NA-I00 by the Generaci\'on de Conocimiento programme and through grant RYC2021-034744-I by the Ram\'on y Cajal programme from Agencia Estatal de Investigaci\'on (Spanish Ministry of Science, Innovation and Universities).}
\maketitle

\section{Introduction and main results}

\subsection{$\hb$-spaces} In this note we study embeddings of other Banach spaces inside the de Branges-Rovnyak spaces $\hb$. The space $\hb$ is the Hilbert space of analytic functions in the unit disk $\D := \{ z \in \mathbb{C} : |z| < 1\}$ with a reproducing kernel $k_b: \D \times \D \to \mathbb{C}$ given by \[ k_b(\lambda, z) := \frac{1 - \conj{b(\lambda)}b(z)}{1-\conj{\lambda} z}. \] and where $b: \D \to \D$ is analytic. Resources on the theory of $\hb$-spaces include Sarason's little book \cite{sarasonbook} and more recent extensive monographs by Fricain and Mashreghi \cite{hbspaces1fricainmashreghi}, \cite{hbspaces2fricainmashreghi}.

Unfortunately, the various methods for construction of the space $\hb$ (for instance, those given in the above references) often leave one wondering what kind of functions $f$ are actually contained in the space, and how to compute the corresponding norms $\| f\|_{\hb}$. We always have $\hb \subset \hil^2$, the latter one being the classical Hardy space. The trivial case $\hb = \hil^2$, with equivalent norms, occurs precisely when $\sup_{z \in \D} |b(z)| < 1$. A celebrated result of Sarason states that the condition for containment in $\hb$ of the set of analytic polynomials $\Po$ is equivalent to the condition \[ \int_\T \log (1-|b(\zeta)|) d\m (\zeta) > -\infty\] where by $d\m$ we denote the Lebesgue measure on the unit circle $\T = \partial \D = \{ z \in \mathbb{C} : |z| = 1\}$ (normalized, for convenience, by the condition $m(\T) = 1$). The above logarithmic integrability condition is well-known to be equivalent to $b$ being a non-extreme point of the unit ball of $\hil^\infty$, the algebra of bounded analytic functions in $\D$. There is an observed dichotomy of properties of $\hb$ depending on if $b$ is extreme or not, and in this note we deal exclusively with \emph{non-extreme} symbols $b$.  

\subsection{A classical inclusion result}
One of Sarason's results from \cite[Theorem 1]{sarason1986doubly} characterizes when the containment $\hil^\infty \subset \hb$ occurs. In order to state his result, we need to first introduce the standard notion of a \textit{Pythagorean mate} $a: \D \to \D$ of $b$. This $a$ is the outer function satisfying the relation \[ |a|^2 + |b|^2 = 1\] on the boundary $\T$. The above equality is interpreted, as usual, in terms of the boundary functions on $\T$ induced by $a$ and $b$. The Pythagorean mate is unique if we impose the normalization condition $a(0) > 0$. The function $\phi = b/a$ is then a member of the Smirnov class $\N^+$ of quotients of bounded analytic functions with outer denominator (see \cite[Chapter II]{garnett}). It is not hard to see that the mapping $b \mapsto \phi = b/a$ implements a bijection between the non-extreme points of the unit ball of $\hil^\infty$ and the Smirnov class $\N^+$. In particular, we may think of any non-extreme space $\hb$ as corresponding to a function $\phi \in \N^+$. 

Sarason's result then reads as follows. 

\begin{thm}{\textbf{(Sarason)}} \thlabel{sarasonHinftyThm} Let $\phi = b/a$. The following two statements are equivalent:
\begin{enumerate}[(i)]
    \item $\hil^\infty \subset \hb$,
    \item $\phi \in \hil^2$.
\end{enumerate}    
\end{thm}

Whenever we have an inclusion of a space $X \subset \hil^2$, with the former space being continuously contained in $\hil^2$, the closed graph theorem implies that we can control the $\hb$-norm in the following way: there exists a constant $C > 0$ such that for each $f$ one has \[ \| f\|_{\hb} \leq C \| f\|_X.\] We interpret both above quantities to be $+\infty$ if $f$ is not contained in the corresponding space. Thus Sarason's result tells us precisely when the $\hb$-norm can be controlled by the supremum norm. Similar matters have been previously investigated in \cite{revcarlesonross}, where reverse Carleson measures for $\hb$-spaces have been studied. The results of this note will characterize symbols $b$ for which the $\hb$-norm is controlled by the norms of the Hardy spaces $\hil^p$, for $p \in [2, \infty]$.

\subsection{Inclusion of $\hil^p$-spaces} The earlier mentioned condition $\| b \|_\infty < 1$ is equivalent to the statement that $\phi = b/a \in \hil^\infty$, which therefore is equivalent to the Banach space equality $\hb = \hil^2$. Together with Sarason's above result, we obtain the two statements \[ \phi \in \hil^\infty \quad \Leftrightarrow \quad \hil^2 = \hb\] and \[ \phi \in \hil^2 \quad \Leftrightarrow \quad \hil^\infty \subset \hb.\] Cast in this way, one is immediately lead to attempt to \textit{``interpolate''} between these two results and derive inclusions $\hil^p \subset \hb$ as a consequence of the containment $\phi \in \hil^{\tilde{p}}$ for some $\tilde{p} > 2$. Here $\hil^p$ is the classical Hardy space consisting of functions analytic in $\D$ for which we have \begin{equation}
    \label{hpnormdef}
    \| f\|_p^p := \sup_{r \in (0,1)} \int_\T |f(r\zeta)|^p d\m(\zeta) < \infty.
\end{equation} Our first main result realizes the mentioned \textit{``interpolation''}.

\begin{mainthm}\thlabel{mainThm1}
    Let $\phi = b/a$, $p \in (2, \infty)$ and $\tilde{p} = \frac{2p}{p-2}$. The following two statements are equivalent:
    \begin{enumerate}[(i)]
    \item $\hil^p\subset \hb$,
    \item $\phi \in \hil^{\tilde{p}}$.
\end{enumerate}       
\end{mainthm}

As a consequence, we know precisely when the $\hb$-norm can be controlled by the $\hil^p$-norms in \eqref{hpnormdef}. Moreover, at the end-point $p = \infty$ our method actually gives an improvement of Sarason's \thref{sarasonHinftyThm}. 

\begin{mainthm}\thlabel{mainThm2}
    The following three statements are equivalent. 
    \begin{enumerate}[(i)]
    \item $\hil^\infty \subset \hb$,
    \item $\BMOA \subset \hb$,
    \item $\phi \in \hil^2$.
\end{enumerate}    
\end{mainthm}

In fact, one may replace $\BMOA$ by any of the spaces $\VMOA$ or the disk algebra $\A$ in the statement of the theorem. Here $\BMOA$ and $\VMOA$ are, respectively, the spaces of analytic functions of bounded and vanishing mean oscillation on $\T$, namely, the dual and the pre-dual of the Hardy space $\hil^1$.

\subsection{Method: a multiplier criterion} Our \thref{mainThm1} and \thref{mainThm2} are consequences of an abstract result relating the containment of a rather general space $X$ inside $\hb$ to the boundedness of a certain multiplication operator. In the setting of our abstract result, we consider Banach spaces $X$ of analytic functions in $\D$ which admit a \textit{Cauchy dual} $X^*$ and which we define precisely in Section \ref{CauchyDualitySection}. In short, a space $X$ admits a Cauchy dual $X^*$ if every bounded linear functional on $X$ is represented by an analytic function $g \in X^*$, and the duality pairing is \begin{equation}
\ip{f}{g} := \lim_{r \to 1^-} \int_\T f(r\zeta)\conj{g(r\zeta)} d\m(\zeta), \quad f \in X, g \in X^*.
\end{equation}
For instance, if $X = \hil^p$, $1 < p < \infty$, then the Cauchy dual $X^*$ equals $\hil^q$, where $q$ is the Hölder conjugate index.

In simplest form, our abstract criterion reads as follows.

\begin{mainthm}\thlabel{mainThm3}
Let $X$ be a Banach space of analytic functions in $\D$ in which the analytic polynomials are dense, and let $\phi = b/a$. The following two statements are equivalent.

\begin{enumerate}[(i)]
    \item The multiplication operator $M_\phi : f \mapsto \phi f$ is bounded from $\hil^2$ to $X^*$.
    \item We have a continuous embedding $X \hookrightarrow \hb$.
\end{enumerate} The compactness of $M_\phi$ is equivalent to the compactness of the corresponding embedding.
\end{mainthm}

In Section \ref{HardyApplicationSection} we appropriately relax the assumption on the density of polynomials in $X$ in order to apply our result to examples such as $X = \BMOA$. In Section \ref{section4} we apply the multiplier criterion to some other types of spaces $X$.

\subsection{Containment of the Dirichlet space} The case of the containment of the Dirichlet space $\Di$ inside a given $\hb$ is a curious one. The space $\Di$ consists of functions satisfying \[ \| f \|^2_{\Di} := \|f\|^2_{2} +  \int_\D |f'(z)|^2 dA(z) <\infty,\] where $dA$ is the area measure on the unit disk, normalized by the condition $A(\D) = 1$. An application of our multiplier criterion in \thref{mainThm3} readily tells us that we have $\Di \subset \hb$ if and only if the measure $|\phi|^2 dA$ is a Carleson measure for $\hil^2$, with the embedding being compact if this measure is a vanishing Carleson measure (see Section \ref{WeighedTaylorSeriesSpacesSection} for the definitions and proof of this claim). This is a rather insatisfactory characterization, and for this reason we spend some portion of this note in studying a specific example. 

Note that according to \thref{mainThm1} the containment of the Hardy spaces $\hil^p$ inside $\hb$ is determined entirely by the magnitude of the boundary values of $\phi$ (or $b$) on $\T$. In the case of the Dirichlet space and the Carleson measure condition, it is not immediately clear if the inner factor $I_b$ of $b$ plays a role, and if so, to what extent. To study this question, we introduce a one-parameter family of Smirnov class functions \[ \phi_c(z) = \frac{1}{(1-z)^c}, \quad z \in \D, \quad c > 0\] and consider \[ \theta(z) = \exp\Big(-\frac{1}{2} \cdot \frac{1+z}{1-z}\Big), \quad z \in \D.\] To $\phi_c$ there corresponds a unique (up to a unimodular constant factor) outer function $b_c : \D \to \D$ which satisfies the boundary value equation \[|\phi_c|^2 = \frac{|b_c|^2}{1-|b_c|^2}.\] In \thref{PhicDirichletContainmentExample} and \thref{ThetaPhicDirichletContainmentProp} below, we use the Carleson measure condition to find out the ranges of a parameter $c$ for which the de Branges-Rovnyak spaces corresponding to the Smirnov class functions $\phi_c$ and $\theta \phi_c$ contain the Dirichlet space. We obtain that the answer differs depending on if the inner factor $\theta$ is present or not, which confirms that the containment of $\Di \subset \hb$ is a matter more delicate than the containments $\hil^p \subset \hb$.

\subsection{Some other related results} Several existing works deal with results connected to ours. The converse to $\hb \subset X$ has been studied by Bellavita and Dellepiane in \cite{bellavita2024embedding}, in the particular case that $X$ is a so-called local Dirichlet space $\Di_\zeta$. A result of Sarason from \cite{sarason1997local} shows that the isometric equality $\hb = \Di_\zeta$ holds for an appropriate choice of $b$ and $\zeta$. A thorough investigation of when isomorphic (not necessarily isometric) equality between $\hb$ and a harmonically weighted Dirichlet space holds has been carried out in \cite{costara2013branges}. It turns out that such an equality holds only in rather specific cases.

\section{Abstract considerations: the multiplier criterion}

This section deals with elementary material. We present first some parts of the theory of $\hb$-spaces which are important in the sequel. Next, we introduce unbounded Toeplitz opeators and Cauchy duals. The section ends with a proof of our multiplier criterion.

\subsection{A linear operator equation}For a function $h \in \hil^\infty$, let $\To{h}$ denote the usual co-analytic Toeplitz operator with symbol $\conj{h}$. That is, if $\L^2(\T)$ is the space of square-integrable functions on the unit circle $\T$, and $P_+: \L^2(\T) \to \hil^2$ is the orthogonal projection, then $\To{h}f = P_+ \conj{h}f$ for $f \in \hil^2$. An analysis of the following linear operator equation in \eqref{hbEquation} will lead to the proofs of our main results.

\begin{prop} \thlabel{hbcontainment} Let $b$ be a non-extreme point of the unit ball of $\hil^\infty$, and $a: \D \to \D$ be the Pythagorean mate of $b$. Then $f\in \hil^2$ is contained in $\hb$ if and only if a solution $f_+ \in \hil^2$ exists to the operator equation 
\begin{equation}
    \label{hbEquation} \To{b}f = \To{a}f_+.
\end{equation} If a solution $f_+$ exists, then it is unique, and we have \[ \| f\|_{\hb}^2 = \|f\|^2_2 + \|f_+\|_2^2.\]
\end{prop}

The result is well-known and forms a basis for the study of non-extreme $\hb$-spaces. For a discussion and a proof, see for instance \cite[Section 23.3]{hbspaces2fricainmashreghi}. 

\subsection{Toeplitz operators with unbounded symbols}

In order to solve the equation \eqref{hbEquation}, one might be tempted to set $\phi = b/a$ and formally solve \[ f_+ = \To{\phi}f,\] which of course doesn't make sense, since the unbounded function $\phi$ cannot at once be interpreted as a well-defined symbol of a Toeplitz operator on $\hil^2$. Nevertheless, appropriately defined Toeplitz operators with such unbounded symbols come in handy in the analysis of \eqref{hbEquation}. They will play a key role throughout this note.

Let $\Po$ denote the set of analytic polynomials, and $\N^+$ denote the Smirnov class of analytic functions $\phi$ in $\D$ which can be expressed as quotients $\phi = c/d$, with $c,d \in \hil^\infty$ and $d$ an outer function (for precise definitions see, for instance, \cite[Chapter II]{garnett}).

\begin{defn} \thlabel{UnboundedToeplitzDef}
     If $\phi \in \N^+$ and $\{\phi_k\}_{k\geq 0}$ is the sequence of coefficients of the Taylor expansion of $\phi$ at $z = 0$, then we define $\To{\phi}: \Po \to \Po$ as the linear operator which, with respect to the basis $\{ z^n\}_{n=0}^\infty$ of $\Po$, has the following matrix representation:
     \[
     \begin{pmatrix}
    \conj{\phi_0} & \conj{\phi_1} & \conj{\phi_2} & \conj{\phi_3}  & \ldots \\
    0 & \conj{\phi_0} & \conj{\phi_1} & \conj{\phi_2}  & \ddots \\
    0 & 0 & \conj{\phi_0} & \conj{\phi_1}  & \ddots \\
    \vdots & \vdots & \vdots & \ddots & \ddots \\
\end{pmatrix}.\]
\end{defn}
In other words, $\To{\phi}$ is the linear extension of the following action on monomials $z^n$ (with slight abuse of notation): \begin{equation}
        \label{TphiDef} \To{\phi}z^n = \sum_{k=0}^n \conj{\phi_{n-k}} z^k.
    \end{equation}    
If $\phi$ is bounded in $\D$, then the above definition coincides on the polynomials with the action of the usual bounded Toeplitz operator $\To{\phi}: \hil^2 \to \hil^2$.
\begin{lem}\thlabel{TalgebraHom} If $\phi, \psi \in \N^+$, then \[ \To{\phi \psi} = \To{\phi}\To{\psi}.\]
\end{lem}
A proof of \thref{TalgebraHom}, which we skip, requires only a short computation using the above matrix representation of the operators. It follows that if $p$ is a polynomial and $\phi = b/a$, then $\To{a}\To{\phi}p = \To{b}p$. In notation of \thref{hbcontainment}, we have \[p_+ = \To{\phi}p\] and \begin{equation}
\label{hbPolynomialNormFormula} \|p\|_{\hb}^2 = \|p\|_2^2 + \|\To{\phi}p\|_2^2.\end{equation} In particular, we deduce from \eqref{TphiDef} that we have \begin{equation}
    \label{hbMonomialNormFormula} \|z^n\|^2_{\hb} = 1 + \sum_{k=0}^n |\phi_k|^2.
\end{equation} This is, of course, all well-known to specialists of the $\hb$-theory.

For certain well-behaved spaces $X$, the relation in \eqref{hbPolynomialNormFormula} immediately implies a criterion for containment of $X \subset \hb$. 

\begin{cor} \thlabel{criterion1lemma}
    Let $X$ be a Banach space continuously contained in $\hil^2$, and $\phi = b/a$. If $X$ contains the algebra of polynomials as a norm-dense subset, then $X \subset \hb$ if and only if there exists a constant $C > 0$ such that \[  \| \To{\phi} p \|_2 \leq C \|p\|_X, \quad p \in \Po.\]
\end{cor}

\begin{proof}
If the above inequality holds, then take any $f \in X$ and a sequence of polynomials $\{p_n\}_n$ which converges to $f$ in the norm of $X$. The continuous containment of $X$ in $\hil^2$ implies that $p_n(z) \to f(z)$ for each $z \in \D$, and $\sup_n \|p_n\|_2 < +\infty$. Our assumption implies also that $\sup_n \| \To{\phi}p_n\|_2 < +\infty$, and then from the $\hb$-norm formula in \eqref{hbPolynomialNormFormula} we conclude that $\sup_n \| p_n\|_{\hb} < +\infty$. Thus we may pass to a subsequence of $\{p_n\}_n$ which converges weakly to some function in $\hb$. By the above stated pointwise convergence, this limit must be $f$, which is thus a member of $\hb$.

The converse implication follows immediately from closed graph theorem and \eqref{hbPolynomialNormFormula}. Indeed, the embedding $X \hookrightarrow \hb$ is a closed operator, and so there exists a constant $C > 0$ such that $\|p\|_{\hb} \leq C\|p\|_X$ for $p \in \Po$. The result now follows from \eqref{hbPolynomialNormFormula},
\end{proof}

\subsection{Cauchy duality} \label{CauchyDualitySection}
The criterion in \thref{criterion1lemma} simply means that $\To{\phi}: \Po \to \Po$ extends to a bounded operator $\To{\phi}: X \to \hil^2$. This condition might be hard to verify. Given some additional structure on $X$, we will be able to rephrase \thref{criterion1lemma} in terms of the boundedness of a multiplication operator which is the adjoint of $\To{\phi}$.

To do so, we will assume that $X$ admits a dual space $X^*$ which is itself a space of analytic functions in $\D$. The duality pairing $\ip{\cdot}{\cdot}$ between the spaces is assumed to have the form \begin{equation}
    \label{CauchyDualEq} \ip{f}{g} := \lim_{r \to 1^-} \int_\T f(r\zeta)\conj{g(r\zeta)} d\m(\zeta), \quad f \in X, g \in X^*.
\end{equation}
Here $d\m$ denotes the Lebesgue measure on the unit circle $\T$, normalized by the condition $m(\T) = 1$. Such duality pairings are usually called \textit{Cauchy pairings}, and $X^*$ is called the \textit{Cauchy dual} to $X$. The Cauchy dual $X^*$ contains the polynomials $\Po$, and this set may or may not be norm-dense in $X^*$ (consider $X = \hil^1, X^* = \BMOA$). In the case that both $f$ and $g$ are members of $\hil^2$, the definition \eqref{CauchyDualEq} reduces to \[ \ip{f}{g} = \ip{f}{g}_2 := \int_\T f\conj{g} \, d\m = \sum_{n=0}^\infty f_n \conj{g_n},\] where $\{f_n\}_n$ and $\{g_n\}$ are the sequences of Taylor coefficients (centered at $z=0$) of the functions $f$ and $g$.
From this observation, and from \eqref{TphiDef}, we easily deduce that \begin{equation}
    \label{ToeplitzPhiMatrixEq} \ip{\To{\phi}z^n}{z^m} = \conj{\phi_{n-m}},
\end{equation}  where we intepret $\phi_k \equiv 0$ for $k < 0$. Assuming for a moment that $\To{\phi}: X \to \hil^2$ is bounded, we consider the adjoint operator $\To{\phi}^*: \hil^2 \to X^*$. For monomials, the action of the adjoint operator is given by \[ \ip{z^n}{\To{\phi}^*z^m} = \ip{\To{\phi}z^n}{z^m} = \conj{\phi_{n-m}},\] from which we deduce (again, slightly abusing notation) that \[ \To{\phi}^*z^m = \sum_{n=0}^\infty \phi_{n-m} z^n = \phi(z)z^m.\] This shows that \[\To{\phi}^*f(z) := M_\phi f(z) = \phi(z)f(z), \quad f \in \hil^2,\] which identifies the adjoint as a multiplication operator $M_\phi: \hil^2 \to X^*$.

\subsection{Multiplier criterion}

Our criterion for containment $X \subset \hb$ reads as follows.

\begin{prop}\thlabel{criterion2lemma} Let $X$ and $X^*$ be as described earlier, with polynomials norm-dense in $X$, and $\phi = b/a$. The following three statements are equivalent.
\begin{enumerate}[(i)]
    \item $X$ is contained in $\hb$.
    \item The co-analytic Toeplitz operator $\To{\phi}: X \to \hil^2$ is bounded.
    \item The multiplication operator $M_\phi: \hil^2 \to X^*$ is bounded.
\end{enumerate}
\end{prop}

\begin{proof}
We have above already proved the equivalence $(i) \Leftrightarrow (ii)$ and the implication $(ii) \Rightarrow (iii)$. It remains to prove that $(iii) \Rightarrow (ii)$, which is routine. Indeed, the adjoint operator $M_\phi^*: X^{**} \to \hil^2$ is bounded, and therefore so is its restriction $M_\phi^*: X \to \hil^2$. We have \[ \ip{M^*_\phi z^n}{z^m}_2 = \ip{z^n}{M_\phi z^m} = \conj{\phi_{n-m}}\] which shows that $M^*_\phi = \To{\phi}$.
\end{proof}

As it stands, \thref{criterion2lemma} cannot be applied to $X = \BMOA$, since the density of polynomials assumption does not hold in this case. However, $X = \VMOA$ satisfies this assumption, and in \thref{ContainmentImprovementProp} below we shall see how to deduce the contaiment $\BMOA \subset \hb$ from a general statement and the containment $\VMOA \subset \hb$.

\subsection{Compact embedding criterion}

The compactness of the embedding $X \subset \hb$ can also be characterized in terms of properties of the operator $\M_\phi: \hil^2 \to X^*$. Note that it is necessary for the embedding $X\ \subset \hil^2$ to be compact if $X \subset \hb$ is to be compact. 

\begin{prop} 
Let $X$ and $X^*$ be as before, and $\phi = b/a$. Assume that the embedding $X \hookrightarrow \hil^2$ is compact. The following three statements are equivalent. 

\begin{enumerate}[(i)]
    \item $X$ is compactly embedded in $\hb$. 
    \item The co-analytic Toeplitz operator $\To{\phi}: X \to \hil^2$ is compact.
    \item The multiplication operator $M_\phi: \hil^2 \to X^*$ is compact.
\end{enumerate}
\end{prop}

\begin{proof} 
The equivalence $(ii) \Leftrightarrow (iii)$ is a well-known property of the adjoint operation. 

Assume that $(i)$ holds. Then we already know from \thref{criterion2lemma} that $\To{\phi}: X \to \hil^2$ is bounded, and we want to show that this operator is also compact. We verify this by showing that $\To{\phi}$ takes weakly convergent sequences to norm convergent ones. By density, it suffices to consider sequences of polynomials. So, take a sequence $\{p_n\}_n$ of polynomials weakly convergent to $0$ in $X$. We have \[ \limsup_{n \to \infty} \|\To{\phi}p_n\|_2 \leq \limsup_{n \to \infty} \|p_n\|_{\hb} = 0\] by \eqref{hbPolynomialNormFormula} and compactness of the embedding $X \hookrightarrow \hb$. We have established $(i) \Rightarrow (ii)$. 

We prove the converse implication $(ii) \Rightarrow (i)$ by showing that a sequence $\{p_n\}_n$ of polynomials which is bounded in the norm of $X$ converges to zero in the norm of $\hb$. Since $X$ is compactly embedded in $\hil^2$, we have $\lim_{n \to \infty} \| p_n\|_2 = 0$. Moreover, the compactness of $\To{\phi}: X \to \hil^2$ implies that $\lim_{n \to \infty} \|\To{\phi} p_n\|_2 = 0$. Now \eqref{hbPolynomialNormFormula} yields convergence of $\hb$-norms to zero.
\end{proof}

\section{Direct applications: Hardy spaces and their duals} \label{HardyApplicationSection}

By \thref{criterion2lemma}, the containment problem $X \subset \hb$ has been reduced to identification of some Cauchy duals and to characterizing the multipliers of $\hil^2$ into the dual. In some cases this is an easy task.

\subsection{Hardy spaces} If $p < 2$, then $\hil^p$ is strictly larger than $\hb$ (which is a subset of $\hil^2$), but if $p \geq 2$, then the containment $\hil^p \subseteq \hb$ might hold. For $p = 2$ equality holds if and only if $\phi = b/a$ is bounded in $\D$. This result goes back to Sarason and can be deduced from a comparison of the reproducing kernels of the involved spaces.

For finite $p > 1$, the Cauchy dual of the Hardy space $\hil^p$ is simply $\hil^q$, where $q = \frac{p}{p-1}$ is the Hölder conjugate index.  According to our multiplier criterion, to understand when $\hil^p \subset \hb$ holds, we need to characterize the space of multipliers from $\hil^2$ to $\hil^q$. The following proposition, well-known and elementary, does this job.

\begin{prop} If $1 < q < 2$, then the multiplication operator $\M_\phi: \hil^2 \to \hil^q$ is bounded if and only if $\phi \in \hil^{\tilde{q}}$, where \[ \tilde{q} = \frac{2q}{2-q}.\]
The operator $\M_\phi: \hil^2 \to \hil^q$ is not compact unless $\phi \equiv 0$.
\end{prop}

\begin{proof}
    The sufficiency of the condition $\phi \in \hil^{\tilde{q}}$ follows readily from Hölder's inequality. Conversely, if $M_\phi: \hil^2 \to \hil^q$ is bounded, then clearly $\phi$ defines also a multiplier from $\L^2(\T)$ to $\L^q(\T)$. Consequently, for some constant $C > 0$ and the Hölder conjugate index $p = \frac{q}{q-1}$, we have \[ \int_\T \phi h \conj{g} d\m \leq C\|h\|_2 \|g\|_p, \quad h \in \L^2(\T), g \in \L^p(\T).\] If $\tilde{p} = \frac{2p}{p-2}$, then one checks readily that each $f \in \L^{\tilde{p}}(\T$) is of the form $f = h\conj{g}$ for some $h \in \L^2(\T)$ and $g \in \L^p(\T)$, with $\|f\|_{\tilde{p}} = \|h\|_2 \|g\|_p$. The above expression then shows that $\phi$ lies in the dual space to $\L^{\tilde{p}}(\T)$, this dual being $\L^{\tilde{q}}(\T)$. Thus $\phi \in \hil^{\tilde{q}}$.

    As for compactness, the sequence $\{z^n\}_n$ is weakly convergent to $0$ in $\hil^2$, yet the norm of $\phi z^n$ in $\hil^q$ remains constant, and so $M_\phi: \hil^2 \to \hil^q$ cannot be compact unless $\phi \equiv 0$.
\end{proof}

We have reached our first main result, describing containment of Hardy spaces in $\hb$.

\begin{cor} \thlabel{HardyContainmentCor} For $p \in (2,\infty)$, we have that $\hil^p \subset \hb$ if and only if $\phi = b/a \in \hil^{\tilde{p}}$, where \[\tilde{p} = \frac{2p}{p-2}.\] The containment $\hil^p \subset \hb$ is never compact.    
\end{cor}

\subsection{Spaces of bounded mean oscillation}
Letting $p \to +\infty$ in \thref{HardyContainmentCor}, one might guess that $\hil^\infty \subset \hb$ if and only if $\phi \in \hil^2$. This fact has been proved by Sarason in \cite{sarason1986doubly}. In fact, much more can be said.

\begin{prop} \thlabel{BMOAContProp}
    Let $\phi = b/a$. The following statements are equivalent:
    \begin{enumerate}[(i)]
        \item $\phi \in \hil^2$,
        \item $\hil^\infty \subset \hb$
        \item $\A := \hil^\infty \cap C(\T) \subset \hb$,
        \item $\VMOA \subset \hb$,
        \item $\BMOA \subset \hb$.
    \end{enumerate}
\end{prop}

We shall give two different proofs of the containment $\BMOA \subset \hb$ given the assumption $(i)$ above. One based on a direct analysis of \eqref{hbEquation}, and one making use of our multiplier criterion developed earlier.

Here is the first proof. Recall that $f \in \BMOA$ if and only if it can be expressed as \[f = P_+u,\] where $u \in \L^\infty(\T)$ and where $P_+$ is the orthogonal projection $\L^2(\T) \to H^2$. Under the assumption that $\phi = b/a \in \hil^2$, the solution $f_+$ to the operator equation \eqref{hbEquation} is simply \[ f_+ = P_+ \conj{\phi} u.\] Indeed, note first that $\conj{\phi} u \in \L^2(\T)$, so that the formula above defines an element $f_+ \in \hil^2$. Next, we verify that \eqref{hbEquation} indeed holds for this choice of $f_+$:
\[   \To{a}(P_+\conj{\phi}u) = P_+(\conj{a \phi}u) = P_+ (\conj{b} u)  = P_+ (\conj{b} P_+u)  = \To{b}f.
\]

This shows that $\BMOA \subset \hb$ if $\phi \in \hil^2$, proving $(i) \Rightarrow (v)$ in \thref{BMOAContProp}.  The converse direction $(v) \Rightarrow (i)$ follows already from Sarason's theorem which asserts the equivalence of $(i)$ and $(ii)$ above, since it is well-known that $\hil^\infty \subset \BMOA$. In fact, that all the other four statements separately imply $(i)$ follows easily from the monomial norm formula in \eqref{hbMonomialNormFormula}. Indeed, statements $(ii), (iv)$ and $(v)$ all separately imply $(iii)$, while this statement together with the closed graph theorem gives us that \[ \|z^n\|_{\hb} \leq C \|z^n\|_{\A} = C\] for some constant $C > 0$. This proves $(i)$, since by \eqref{hbMonomialNormFormula}, we have \[\|\phi\|_2 + 1= \lim_{n \to \infty} \|z^n\|_{\hb}.\] This concludes the proof of \thref{BMOAContProp}. 

Here is a second proof of the containment $\BMOA \subset \hb$ under the condition $(i)$. We start by noting that if we set $X = \VMOA$, then it is well-known from the works of Fefferman and Sarason that $X^* = \hil^1$, and so by our \thref{criterion2lemma}, the containment $\VMOA \subset \hb$ is equivalent to $\phi$ being a multiplier from $\hil^2$ into $\hil^1$. We are assuming that $\phi \in \hil^2$, and it is well-known that $\hil^2 \cdot \hil^2 = \hil^1$. Thus we obtain $\VMOA \subset \hb$. Now, it is a fact that each function $f \in \BMOA$ can be approximated by a sequence of polynomials $\{p_n\}_n$ in the following way: $p_n(z) \to f(z)$ for all $z \in \D$, and \[\sup_n \|p_n\|_{\VMOA} = \sup_n \|p_n\|_{\BMOA} < \infty\] (for instance, if $f = P_+u$, $u \in \L^\infty(\T)$, then we may take $p_n = P_+ u_n$, where $u_n$ are the Ces\`aro means of $u$). But then, $\sup_n \|p_n\|_{\hb} < \infty$ follows from the closed graph theorem and the containment $\VMOA \subset \hb$, and so $f$ can be identified as a weak limit of the polynomials $\{p_n\}_n$, as it was done in the proof of \thref{criterion1lemma}. Thus $f \in \hb$, and we have another proof of the implication $(i) \Rightarrow (v)$ in \thref{BMOAContProp} above. 

Clearly, this last argument can be formulated in a more general setting. For any Banach space of analytic functions in $\D$ we assume that the evaluations at points of $\D$ are bounded in the norm of $X$.

\begin{prop} \thlabel{ContainmentImprovementProp} Let $X$ and $\hil$ be, respectively, a Banach and a Hilbert space of analytic functions in $\D$. If $X \subset \hil$, then we also have $\conj{X} \subset \hil$, where \[ \conj{X} := \{ f \in \hd : f(z) = \lim_{z \in \D} p_n(z), p_n \in \Po, \, \sup_{n} \|p_n\|_X < \infty \}. \]
\end{prop}

\section{More complicated matters}
\label{section4}

Let $\Di$ denote the Dirichlet space. Explicit and readily-verifiable conditions ensuring $\Di \subset \hb$ are more difficult to reach than their analogues for Hardy spaces. 

\subsection{Weighted spaces of Taylor series} \label{WeighedTaylorSeriesSpacesSection} The condition for the containment $\Di \subset \hb$ is the same as a corresponding condition for a larger family of Hilbert spaces $\hil^2(w)$ to which the Dirichlet space belongs. Here $w = \{ w_n\}_{n\geq 0}$ is an non-decreasing sequence of positive numbers, and \[ \hil^2(w) := \{ f \in \hd : \|f \|_{\hil^2(w)} < \infty \}, \] where the norm is defined in terms of the sequence $\{f_n\}_{n \geq 0}$ of Taylor coefficients of $f$, centered at $z = 0$, by the expression \[ \| f\|^2_{\hil^2(w)} := \sum_{n=0}^\infty w_n |f_n|^2 < \infty.\] 
The Cauchy dual of $\hil^2(w)$ is a space $\hil^2(w^*)$ of the same type, with \[ w^* := \{ 1/w_n \}_{n \geq 0}.\] Here we need to implicitly assume that $\lim_{n \to \infty} w_n^{1/n} = 1$ in order for $\hil^2(w)$ and $\hil^2(w^*)$ to consist of power series with radius of convergence at least $1$ (that is, to consist of functions analytic in $\D$). 

Our multiplier criterion in \thref{criterion2lemma} is particularly useful if the norm on $\hil^2(w^*)$ can be realized using an equivalent integral norm of the form \begin{equation}
    \label{h2wStarNormEquivalenceEquation}\| f\|^2_{\hil^2(w^*)} \simeq \int_\D |f(z)|^2 G(|z|) dA(z) = \sum_{n=0}^\infty |f_n|^2 G_n 
\end{equation} where $G: (0,1) \to (0,\infty)$ is some positive function and \begin{equation}
    \label{momentsG} G_n := 2 \int_0^1 G(r) r^{2n+1} dr
\end{equation} is the sequence of moments of $G$. Above ``$\simeq$'' means that the two expressions are of comparable size, independently of $f$. We have the following consequence of our multiplier criterion.

\begin{prop} \thlabel{CarlesonMeasureContainmentCharProp} Let $w$ be a weight sequence for which \eqref{h2wStarNormEquivalenceEquation} holds for some $G$. If $\phi = b/a$, then $\hil^2(w) \subset \hb$ if and only if \[d\mu(z) = |\phi(z)|^2G(|z|)dA(z)\] is a Carleson measure. The embedding $\hil^2(w) \hookrightarrow \hb$ is compact if and only if $d\mu(z)$ is a vanishing Carleson measure.
\end{prop}

Here, as usual, a positive measure $\mu$ on the disk $\D$ is a Carleson measure if 
\[ \int_\D |f(z)|^2 d\mu(z) \leq C\|f\|_2^2, \quad f \in \hil^2\] and where $C > 0$ is some positive constant. In other words, $\mu$ is a Carleson measure if we have a continuous embedding $\hil^2 \hookrightarrow \L^2(\mu)$. It is a well-known fact that Carleson measures are precisely those for which, for some $C > 0$, we have a bound
\begin{equation}
    \label{CarlesonMeasCond} \mu\big( S(\theta_0, h)\big) \leq Ch,
\end{equation} and where $S(\theta_0, h)$ is a (curvlinear) Carleson square given by \[ S(\theta_0, h) = \{ z = re^{i\theta} \in \D: |\theta - \theta_0| < h/2, 1-r < h \}, \quad \theta_0 \in [0, 2\pi), h \in (0,1).\] A \textit{vanishing} Carleson measure is one where the right-hand side of \eqref{CarlesonMeasCond} is improved to $o(h)$. The vanishing condition is well-known to be equivalent to compactness of the embedding $\hil^2 \hookrightarrow \L^2(\mu)$.

\begin{proof}[Proof of \thref{CarlesonMeasureContainmentCharProp}] The multiplier criterion tells us that the containment is equivalent to \[ \int_\D |f(z)\phi(z)|^2 G(|z|)dA(z) \leq C \|f\|_{\hil^2},\] which is precisely the Carleson measure condition for $d\mu(z) = |\phi(z)|^2 G(|z|) dA(z)$. The compactness statement is established similarly. 
\end{proof}

\begin{example}
    By setting $G_c(|z|) = \exp\Big( - \frac{c}{1-|z|}\Big)$ for $c > 0$ we obtain a scale of weight sequences $w^c = \{w^c_n\}_{n \geq 0}$ and corresponding spaces $\hil^2(w)$. A rather messy computation of the corresponding moments \eqref{momentsG} of $G_c$ reveals that \begin{equation}
        \label{GevreyUnion} \mathcal{G} = \bigcup_{c > 0} \hil^2(w_c) 
    \end{equation} where $\mathcal{G}$ is the so-called \textit{Gevrey class} consisting of functions $m$ analytic in $\D$ which have a Taylor series expansion $m(z) = \sum_{n = 0}^\infty m_n z^n$ satisfying \[ m_n = O\big( \exp( -c \sqrt{n}) \big)\] for some $c > 0$. By a deep result of Davis and McCarthy from \cite{davis1991multipliers}, the class $\mathcal{G}$ consists precisely of those functions which are multipliers simultaneously for all non-extreme spaces $\hb$. In particular, $\mathcal{G} \subset \hb$ for every non-extreme $b$. We may derive this weaker conclusion immediately from our criterion. Indeed, every function $\phi = b/a \in \N^+$ satisfies the estimate \[ \log|\phi(z)| = o\big( (1-|z|)^{-1}), \quad z \in \D\] See, for instance, \cite{yanagihara1974mean}. But then $|\phi(z)|^2G_c(|z|)$ is bounded for every $c > 0$, so clearly \[ \int_\D |h(z)|^2|\phi(z)|^2 G_c(|z|) dA(z) \leq B \int_\D |h(z)|^2 dA(z) \leq B \|h\|_{\hil^2}, \quad h \in \hil^2,\] where $B > 0$ is some constant depending on $c$. Thus $\hil^2(w^c) \subset \hb$ for every $c > 0$ and every non-extreme $b$, from which it follows by \eqref{GevreyUnion} that $\mathcal{G} \subset \hb$.
\end{example}

\subsection{Dirichlet space} \label{DirichletSubsection} The Dirichlet space $\Di$ corresponds to $\hil^2(w)$ where $w = \{ w_n\}_n$ and $w_n = n+1$. The Cauchy dual of $\Di$ is isometrically equal to the usual (unweighted) Bergman space, which we will denote by $\B$ and which we recall consists of those functions analytic in $\D$ for which
\begin{equation}
    \label{BergmanSpaceDef}
    \|f\|^2_{\B} := \int_\D |f(z)|^2 dA(z) = \sum_{n=0}^\infty \frac{|f_n|^2}{n+1 }< \infty.
\end{equation} Thus $\Di^* = \hil^2(w^*)$ has the form \eqref{h2wStarNormEquivalenceEquation} with $G \equiv 1$. In this case, our Carleson measure condition says that $\Di \subset \hb$ if and only if $|\phi|^2 dA$ is a Carleson measure.

\begin{example} \thlabel{PhicDirichletContainmentExample} Let $c > 0$ and define $b:= b_c : \D \to \D$ by the equation \begin{equation}
    \label{phicDef} \phi_c(z) = \frac{b(z)}{a(z)} = \frac{1}{(1-z)^c}, \quad z \in \D.
\end{equation}
    A simple verification of the Carleson condition in \thref{CarlesonMeasureContainmentCharProp} establishes the following.    
    \begin{enumerate}[(i)]
        \item If $c < 1/2$, then $\hb$ contains $\Di$, and the embedding is compact.
        \item If $c = 1/2$, $\hb$ contains $\Di$, embedding not being compact.
        \item If $c > 1/2$, then $\hb$ does not contain $\Di$.
    \end{enumerate}    
\end{example}

In the next section we will study further the above example, and show precisely how the picture changes if the above $\phi_c$ is replaced by $\theta \phi_c$, where $\theta$ is an appropriate inner function.

\section{A case study: more on containment of the Dirichlet space}

Let us illustrate further how different the Carleson measure criterion appearing in Section \ref{DirichletSubsection} is from that for the Hardy space appearing in \thref{HardyContainmentCor}. In the latter case, only the boundary values of $\phi$ play a role, while in the former case, values inside the disk seem to matter. In particular, a question arises: \emph{to what extent does the inner factor of $\phi = b/a$ (equivalently, inner factor of $b$) play a role in the corresponding containment?} Below, we work out an example.

\label{innerFactorEffectSection}

\subsection{An inner factor that helps} We will extend \thref{PhicDirichletContainmentExample} and prove that the inner factor of $b$ indeed may play a critical role in the containment of the Dirichlet space $\Di$ in $\hb$. We will focus on the singular inner function corresponding to a Dirac measure $\delta_1/2$ at the point $\zeta = 1$: \begin{equation}
    \label{thetaEq} \theta(z) = \exp\Big(-\frac{1}{2} \cdot \frac{1+z}{1-z}\Big), \quad z \in \D.
\end{equation} We will prove the following proposition.

\begin{prop}\thlabel{ThetaPhicDirichletContainmentProp}
    Let $\phi_c$ and $b = b_c$ be as in \thref{PhicDirichletContainmentExample}, and let $\theta$ be given by \eqref{thetaEq}.
    \begin{enumerate}[(i)]
        \item If $c < 1$, then $\hil(\theta b)$ contains $\Di$, and the embedding is compact.
        \item If $c = 1$, then $\hil(\theta b)$ contains $\Di$, the embedding not being compact.
        \item If $c > 1$, then $\hil(\theta b)$ does not contain $\Di$.
    \end{enumerate}
\end{prop}

The choice of the weight $1/2$ in the definition of $\theta$ in \eqref{thetaEq} is not important, and our results hold true for any positive weight used to define a similar singular inner function. With our choice, we have the convenient formula \[ |\theta(z)|^2 = \exp\Big( -\frac{1-|z|^2}{|1-z|^2}\Big), \quad z \in \D.\] 

The above result should be compared to \thref{PhicDirichletContainmentExample}. The point is that the functions $\phi_c$ appearing in \thref{PhicDirichletContainmentExample} correspond, for $c > 1/2$, to measures $|\phi_c|^2 dA$ which violate the Carleson condition for small boxes containing the point $\zeta = 1$ in their closure. However, multiplying $|\phi_c|dA$ by the factor in \eqref{thetaEq} remedies the situation if $c \in (1/2, 1]$ (and only in that range). This happens because of the exponential non-tangential decay of $|\theta(z)|^2$ when $z$ tends to $1$ inside a given \textit{Stolz angle} with vertex at $\zeta = 1$, given by \begin{equation}
    \label{StolzAngle} \Gamma_\alpha := \Big\{ z \in \D : \frac{1-|z|}{|1-z|} \geq \alpha \Big\}, \quad \alpha \in (0,1).
\end{equation}
We will carefully justify this assertion below.

\subsection{Growth properties of multipliers between $\hil^2$ and $\B$} 

In \cite[Theorem 3.1]{stegenga1980multipliers}, Stegenga established an elegant characterization of the functions $\phi$ which are multipliers between $\hil^2$ and the Bergman space $\B$ in terms of the boundary values of the primitive of $\phi$. However, his result doesn't seem to help our particular analysis. Below we will only really need some simple necessary conditions on $\phi$ to be such a multiplier. The following basic results were known to Stegenga and have also been observed in \cite{feldman1999pointwise}. For a proof, one may consult \cite{feldman1999pointwise}.

\begin{prop} \thlabel{GrowthPropMultipliers}
    Let $\phi: \D \to \mathbb{C}$ be an analytic function.
    \begin{enumerate}[(i)]
        \item For $\phi$ to be a multiplier from $\hil^2$ into $\B$, it is necessary that \[ |\phi(z)| = O\big( (1-|z|)^{-1/2}\big).\]
        \item For $\phi$ to be a compact multiplier from $\hil^2$ into $\B$, it is necessary that \[ |\phi(z)| = o\big((1-|z|)^{-1/2}\big).\]
    \end{enumerate}
\end{prop}

\begin{cor}
    Let $\phi_c(z)$ be as in \eqref{phicDef} and $\theta$ be as in \eqref{thetaEq}.
    \begin{enumerate}[(i)]
        \item If $c > 1$, then $\theta \phi_c$ is not a multiplier between $\hil^2$ and $\B$.
        \item If $c = 1$, then $\theta \phi_c = \theta \phi_1$ is not a compact multiplier between $\hil^2$ and $\B$.
    \end{enumerate}
\end{cor}

\begin{proof}
To prove the corollary, we need only to check that the necessary conditions stated in \thref{GrowthPropMultipliers} are violated. To see this, we consider the following level set \[ C_t = \Big\{ z \in \D : |\theta(z)|^2 = t \Big\}, \quad t \in (0,1)\] Setting $s = -\log (t)$, we readily compute that \[ C_t = \Big\{ z = x+iy \in \D : \Big(x-\frac{s}{1+s}\Big)^2 + y^2 = \frac{1}{(1+s)^2} \Big\} \] which is a circle inside $\D$ tangent to $\T$ at the point $\zeta = 1$. For $z = x+iy \in C_t$ we have \[ |1-z|^2 = \frac{1-|z|^2}{s}\] which implies \[ |\theta(z) \phi_c(z)| = \frac{\sqrt{ts^c}}{(1-|z|^2)^{c/2}}, \quad z =x+iy \in C_t.\] If $t$ remains fixed and $c = 1$, then our above estimate implies that $(1-|z|)^{1/2}|\theta \phi_c(z)|$ is constant as $z \to 1$ along $C_t$, and therefore $(ii)$ of \thref{GrowthPropMultipliers} shows that $\theta \phi_1$ cannot be a compact multiplier between $\hil^2$ and $\B$. If $c > 1$, then instead $(1-|z|)^{1/2}|\theta \phi_c(z)|$ grows to infinity as $z \to 1$ along $C_t$, and so now instead $(i)$ of \thref{GrowthPropMultipliers} shows that $\theta \phi_c$ cannot be a multiplier between $\hil^2$ and $\B$.
\end{proof}

It follows from part $(i)$ of the above corollary that $(iii)$ of \thref{ThetaPhicDirichletContainmentProp} holds. Moreover, we know that the containment in $(ii)$ in \thref{ThetaPhicDirichletContainmentProp}, whether it holds or not, at least cannot be compact. We shall now verify the corresponding Carleson measure conditions and prove the remaining part of \thref{ThetaPhicDirichletContainmentProp}.

\subsection{Dyadic Carleson squares} The remaining part of the proof of \thref{ThetaPhicDirichletContainmentProp} requires a computation which is unfortunately a bit messy. To somewhat remedy this, we will use a dyadic system of Carleson squares in $\D$.

If $\mu$ is a positive measure on $\D$, then to verify the Carleson measure condition in \eqref{CarlesonMeasCond} for some $C > 0$ it suffices to verify that there exists a constant $C' > 0$ such that \begin{equation}
    \label{DyadicCarlesonMeasCond} \mu(S_{n,k}) \leq C'2^{-n}
\end{equation} for each \textit{dyadic Carleson square} $S_{n,k}$. We define these as \begin{equation} \label{DyadicCarlesonSqDef} S_{n,k} = \{ z = re^{i\theta} \in \D : \theta \in I_{n,k} , 1-r \leq 2^{-n} \}, \quad n \in \mathbb{N}, \, k \in \{\pm 1, \pm 2, \ldots, \pm 2^{n-1}\}, \end{equation} where $I_{n,k}$ is the dyadic subarc of $\T$ given by \[I_{n,k} =  \{ e^{i\theta} :  2\pi (k-1) 2^{-n} \leq \theta \leq 2\pi k2^{-n} \}\]  for positive $k$, and by \[ I_{n,k} = \{ z \in \D : \conj{z} \in I_{n,-k} \}\] for negative $k$. Indeed, every Carleson square $C(\theta_0, h)$ is contained in the union of at most two such dyadic Carleson squares of side-length at most $2h$, and so \eqref{DyadicCarlesonMeasCond} implies \eqref{CarlesonMeasCond} with $C = 4C'$. For the same reason, to verify the vanishing Carleson measure condition for $\mu$, it suffices to show the left-hand side in \eqref{DyadicCarlesonMeasCond} is of order $o(2^{-n})$.

One obvious geometric property of the dyadic Carleson system that we shall use is the following: for some $\alpha \in (0,1)$, the corresponding Stolz angle $\Gamma_\alpha$ in \eqref{StolzAngle} has such a wide opening that the containment \begin{equation}
    \label{Sn1ContainmentRelation} S_{n,1} \subset \left( \Gamma_\alpha \cup \Big( \cup_{m > n} S_{m,2}\Big) \right)
\end{equation} is satisfied. This containment is quite obvious geometrically, and algebraically we may verify it in the following way. Assume that $z = re^{i\theta} \in S_{n,1}$, but $z \not\in \cup_{m > n} S_{m,2}$. We shall show that $z \in \Gamma_\alpha$ for some $\alpha$ which we shall compute explicitly. There exists a smallest integer $m > n$ for which $z \not\in S_{m,1}$. Then, by assumption, we have that $z \in S_{m-1, 1} \setminus \big(S_{m,1} \cup S_{m,2} \big)$. By definition of the dyadic Carleson squares in \eqref{DyadicCarlesonSqDef}, this tells us that \[ 2^{-m} < 1 - r = 1-|z| < 2^{-m+1}\] and \[ |1-z| < |1-e^{i\theta}| + |e^{i\theta} - z| \leq 2\pi 2^{-m+1} + 2^{-m+1}.\] Thus \[ \frac{1-|z|}{|1-z|} \geq \frac{1}{4\pi + 2} := \alpha. \] Hence $z \in \Gamma_\alpha$.

\subsection{Proof of \thref{ThetaPhicDirichletContainmentProp}}
First, we estimate the $\mu$-measure of the dyadic squares $S_{n,k}$ which do not contain $\zeta = 1$, i.e, we deal with $S_{n,k}$ for $|k| \geq 2$.

\begin{lem} \thlabel{CarlesonEstimate1}
    Let $c\in (0,1]$ and \begin{equation}
        \label{muMeasureEq}d\mu_c(z) = \frac{\exp\Big( - \frac{1-|z|^2}{|1-z|^2}\Big)}{|1-z|^{2c}}dA(z) = |\theta(z)\phi_c(z)|^2 dA(z).
    \end{equation} If $|k| \geq 2$, then \[ \mu_1(S_{n,k}) = O(2^{-n})\] while for $c \in (0,1)$, 
    \[ \mu_c(S_{n,k}) = o(2^{-n}).\] 
\end{lem}

\begin{proof} Let $h := 2^{-n}$. First we treat the case $c=1$. By symmetry, we may assume that $k$ is positive. It is geometrically evident that there exists a constant $D > 1$ which is independent of $n$ and $k$ such that for $z \in S_{n,k}$ we have that \[ \frac{hk}{D} \leq |1-z| \leq Dhk, \quad k \geq 2.\] It follows that \[ \frac{\exp\Big( - \frac{1-|z|^2}{|1-z|^2}\Big)}{|1-z|^{2}} \leq \frac{D^{2}}{h^{2} k^{2}}\exp\Big( - \frac{1-|z|^2}{(Dkh)^2}\Big), \quad z \in S_{n,k}.\] Therefore, using polar coordinates, we readily obtain that 
    \begin{align*}
    \mu_1(S_{n,k}) &\leq \frac{D^{2}}{h^{2} k^{2}}\int_{S_{n,k}} \exp\Big( - \frac{1-|z|^2}{(Dkh)^2}\Big) dA(z) \\ &=
    \frac{D^{2}}{\pi h^{2} k^{2}}\int_{I_{n,k}} \int_{1-h}^1 \exp\Big( - \frac{(1-r)(1+r)}{(Dkh)^2}\Big) r\, dr\, d\theta \\
    &\leq \frac{D^{2}}{\pi  h k^{2}} \int_0^h \exp\Big( \frac{-x}{(Dkh)^2}\Big) dx \\
    &\leq \frac{D^4 h}{\pi}.
    \end{align*} Between the second and third lines we used that $I_{n,k}$ has length $h$. This gives the desired estimate for $c = 1$. If $c \in (0,1)$, then \[\mu_c(S_{n,k}) = \int_{S_{n,k}} |1-z|^{2-2c} d\mu_1(z).\] Fix $\delta > 0$. For $n$ large enough, either $S_{n,k}$ is fully contained within a disk around $\zeta = 1$ of radius $\delta$, or it does not intersect a disk around $\zeta = 1$ of radius $\delta/2$. In the first case, the above estimate, and our result for $c = 1$, gives us \[\mu_c(S_{n,k}) \leq \delta^{2-2c}\frac{D^4 h}{\pi}.\] In the second case, the formula \eqref{muMeasureEq} immediately gives \[ \mu_c(S_{n,k}) \leq \frac{2^{2c}}{\pi \delta^{2c}} h^2.\] It follows that $\mu_c(S_{n,k}) = o(h)$.
\end{proof}

\begin{lem} \label{CarlesonEstimate2} With $\mu_c$ as in \eqref{muMeasureEq}, we have \[ \mu_1(S_{n,1}) = \mu_1(S_{n,-1}) = O(2^{-n})\] and \[ \mu_c(S_{n,1}) = \mu_c(S_{n,-1}) = o(2^{-n}), \quad c \in (0,1).\]    
\end{lem}

\begin{proof}
  The estimate for $c \in (0,1)$ follows from the estimate for $c=1$ in the same way as above in the proof of \thref{CarlesonEstimate1}, so we treat only the latter case. Set, again, $h := 2^{-n}$. By the exponential decay of $|\theta(z)|$ as $z \to 1$ inside the Stolz angle $\Gamma_\alpha$, we certainly have \[ \sup_{z \in \Gamma_\alpha} |\theta(z) \phi_c(z)|^2 < D\] for some constant $D > 0$. Thus, using \eqref{Sn1ContainmentRelation} and \thref{CarlesonEstimate1} in the form $\mu_1(S_{m,2}) \leq A \cdot 2^{-m}$ for some $A > 0$, we obtain \[\mu_1(S_{n,1}) \leq Dh^2 + \sum_{m > n} \mu_1(S_{m,2}) \leq
      Dh^2 + A \sum_{k = 1}^\infty \frac{h}{2^k} =O(h).\] \end{proof}

The two above lemmas show that the measure $\mu_c$ defined in \eqref{muMeasureEq} is a Carleson measure for $c = 1$, and a vanishing Carleson measure for $c \in (0,1)$. Thus, for $c \in (0,1)$ the functions $\theta \phi_c$ are compact multipliers between $\hil^2$ and the Bergman space $\B$, while for $c = 1$, $\theta \phi_c$ is a bounded but not compact multiplier between the same spaces. An application of our multiplier criterion in \thref{criterion2lemma} finishes the proof of \thref{ThetaPhicDirichletContainmentProp}.

\section{Further remarks}

We conclude with a few observations about directions in which to follow with research from here.

\subsection{Other equivalences and embeddings.} In Proposition \ref{BMOAContProp}, we showed that it is equivalent for $\hb$ to contain any of $\hil^\infty$, $\BMOA$, $\VMOA$ or $\A$. It seems reasonable to ask where is the limit of these equivalences, that is, what is the largest (or smallest) space whose containment in $\hb$ is still necessary (respectively, sufficient) for containment of $\hil^\infty$. One can also expand on the results above from the plethora of theorems known regarding embeddings between analytic function spaces. A good and recent reference regarding such embeddings is \cite{Llinares}.

\subsection{Multiplying the symbol by other types of functions.} The effects of the inner part noted in Section \ref{innerFactorEffectSection} could in fact be due to choosing an inadequate generalization of the situation for the embeddings of $\hil^p$ or weighted Bergman spaces. Indeed, multiplying $b$ by any function $\phi$ that is bounded in modulus by 1 will at least preserve all the corresponding embeddings, since $\hil(b) \subset \hil(\phi b)$. A relevant question is, therefore, what functions \emph{improve} the embeddings. Contractive divisors are natural candidates there. Such functions admit several definitions, and play an essential role in extremal problems as well as in the study of invariant subspaces for the shift operator. One can for example say $\phi$ is \textit{Bergman-inner} (or a \emph{contractive divisor}) if it has norm 1 and $\phi \perp z^k \phi$ for all $k \geq 1$ (where the orthogonality is taken in $\B$).

There is a rather interesting connection between embeddings $\Di \subset \hb$ and Bergman-inner functions. It is well-known that such a function induces a bounded multiplication operator $M_\phi: \hil^2 \to \B$ (see, for instance, \cite[Theorem 3.3]{hedenmalmbergmanspaces}). It is also known that if $\theta$ is a singular inner function which is not \textit{cyclic} in $\B$, in the sense that the smallest closed subspace $[\theta]$ of $\B$ which contains $\theta$ and is invariant for the multiplication operator $M_z: f(z) \mapsto z f(z)$ is not the whole space, then any non-zero function $\phi$ in the one-dimensional subspace $[\theta] \ominus M_z[\theta]$ is Bergman-inner, and moreover we have $\phi \in \N$, the Nevanlinna class of functions expressible as a quotient of two bounded functions in $\D$ (see \cite[Theorem 3.3]{hedenmalm1996beurling}. Thus $\phi = \frac{\theta b}{Sa}$ for some Pythagorean pair of outer functions $a$, $b$, and $S$ singular inner (it is plausible that the singular factor $S$ in the denominator is always trivial, but to the best of the authors' knowledge this property of Bergman-inner functions has not yet been established in the existing literature). Then $M_{S\phi} : \hil^2 \to \B$ is also bounded, $S\phi \in \N^+$, and so our \thref{criterion2lemma} implies that $\Di \subset \hil(\theta b)$. We have in this way associated a space $\hil(\theta b)$ containing $\Di$ to each singular inner function $\theta$ which is not cyclic in $\B$. A characterization of these inner functions has been established in the deep works of Korenblum in \cite{korenblum1977beurling} and Roberts in \cite{roberts1985cyclic}. Given our investigations in Section \ref{innerFactorEffectSection} on effects of inner factors, a natural question is the following.

\begin{question*}
    Let $\theta$ be singular inner, and associate $\hil(\theta b)$ to $\theta$ as above. Then $\Di \subset \hil(\theta b)$. But do we have $\Di \subset \hb$?
\end{question*}

From \cite[Theorem 3.7]{hedenmalm1996beurling} we deduce that if $\theta$ is given by \eqref{thetaEq}, then the above construction presents us with the Bergman-inner function $\phi$ which satisfies \[\frac{\phi(z)}{\theta(z)} = \frac{b(z)}{a(z)} = \frac{1}{1-z} + 1, \quad z \in \D,\] from which it easily follows that $\Di \not\subset \hb$ (recall \thref{PhicDirichletContainmentExample}). So the answer to the above question is negative in this particular case. Is this typical? If so, then we will have obtained a family of examples in which various singular inner factors $\theta$ are responsible for the containment $\Di \subset \hil(\theta b)$.

\subsection{Multipliers.} The literature on multipliers between analytic function spaces is vast and our results could certainly benefit from an extensive look at multiplier properties in $\hb$ spaces. A likely ally is the theory of multipliers for model spaces developed in \cite{FHR}.


\end{document}